\documentclass[11pt,twoside, a4paper, english, reqno]{amsart}
\usepackage{graphicx, amsmath, varioref, amscd, amssymb,color, bm, stmaryrd, amsthm}
\usepackage{fancybox}
\usepackage[pdftex, colorlinks=true, backref]{hyperref}
\usepackage{algorithm}
\usepackage{algpseudocode}
\usepackage{eso-pic}
\usepackage{color}
\usepackage{type1cm}
\usepackage{bbm}
\usepackage{mathbbol}

\numberwithin{equation}{section}
\allowdisplaybreaks

\newtheorem{theorem}{Theorem}[section]
\newtheorem{lemma}[theorem]{Lemma}
\newtheorem{corollary}[theorem]{Corollary}

\theoremstyle{definition}
\newtheorem{definition}[theorem]{Definition}

\theoremstyle{remark}
\newtheorem{remark}[theorem]{Remark}

\newcommand{\Div}{\operatorname{div}}

\newcommand{\Grad}{\nabla}

\newcommand{\bi}{\underline{i}}
\newcommand{\bj}{\mathbbm{j}}
\newcommand{\oxe}{\overline{u}_{1}}
\newcommand{\oxz}{\overline{u}_{2}}

\newcommand{\weakstar}{\overset{\star}\rightharpoonup}

\newcommand{\Dt}{\Delta t}

\newcommand{\Om}{\ensuremath{\Omega}}

\newcommand{\vt}{\vartheta}

\newcommand{\eps}{\epsilon}

\newcommand{\equationbox}[2]{\begin{center}

\fbox{
\begin{minipage}{0.96\textwidth}
#1
\end{minipage}
}
\end{center}
}

\begin{document}

\title[A new angular momentum method for wave maps]{A new angular momentum method for\\ computing  wave maps into spheres}

\author[Karper]{Trygve K. Karper}\thanks{This research was funded by the Research Council of Norway.}

\address[Karper]{\newline
Department of Mathematical Sciences, Norwegian University of Science and Technology (NTNU), Norway.}
\email[]{\href{karper@gmail.com}{karper@gmail.com}}
\urladdr{\href{http://folk.uio.no/~trygvekk}{folk.uio.no/\~{}trygvekk}}

\author[Weber]{Franziska Weber}

\address[Weber]{\newline
Department of Mathematics, University of Oslo, Norway.}
\email[]{\href{franziska.weber@cma.uio.no}{franziska.weber@cma.uio.no}}

\date{\today}



\maketitle
\begin{abstract}
In this paper, we present and analyze a new finite difference method for computing three dimensional wave maps into spheres.
By introducing the angular momentum as an auxiliary variable, 
we recast the governing equation as a first order system.
For this new system, we propose a discretization 
that  conserves both the \emph{energy} and the \emph{length constraint}.
The new method is also fast requiring 
 only $N\log N$ operations at each time step.
Our main result is that the method converges to a weak solution 
as discretization parameters go to zero.
The paper is concluded by numerical experiments 
demonstrating convergence of the method and its ability to predict 
finite time blow-up.

\end{abstract}

\section{Introduction}
The purpose of this paper is to develop a new 
numerical method for computing wave maps.
By wave maps, we here  mean vectors $d = [d_1,d_2,d_3]^T$ satisfying 
the following constrained wave equation:
\begin{equation}\label{eq:3Dwave}
    \begin{split}
            d_{tt} - \Delta d &= \gamma d, \quad
            |d| = 1,\quad \text{ in $(0,\infty)\times \Om$}.
    \end{split}
\end{equation}
Here, $\Om \subset \mathbb{R}^n$, $n=2,3$, is either assumed 
to be the unit box $\Om = [0,1]^n$ or it is assumed that $\Om = \mathbb{T}^n$, 
where $\mathbb{T}^n$ is the n-dimensional torus. In the first case, \eqref{eq:3Dwave}
is augmented with homogenous Neumann boundary conditions.

The $\gamma$ appearing in \eqref{eq:3Dwave} is a Lagrange multiplier 
enforcing the constraint $|d|=1$. In particular, by dotting 
\eqref{eq:3Dwave} with $d$ and using that $|d|=1$, one finds that
$$
\gamma = |\Grad d|^2 - |d_t|^2.
$$
Thus, \eqref{eq:3Dwave} is in this sense highly nonlinear which 
in turn obscures the task of developing conservative numerical methods. 
Moreover, in three spatial dimensions, 
it is known that solutions of \eqref{eq:3Dwave} may blow-up \cite{SS}.
Specifically, there is initial data for which the gradient $\Grad d$ develops singularities
in finite time. Thus, solutions of the wave map equation are not smooth.
We will return to the issue of blow-up in the  numerical section
(Section 5).

The literature on numerical methods  
for  \eqref{eq:3Dwave} seems to be confined 
to a handful of results. 
In the papers \cite{BPS, BFP, B}, the authors develops convergent splitting and relaxation methods. 
With these methods, \eqref{eq:3Dwave} is either
solved iteratively, using one evolution step and one projection step onto the sphere,
or the constraint $|d|=1$ is relaxed altogether. 
In the paper \cite{BLP},  the wave map equation \eqref{eq:3Dwave} is computed
using an approximate Lagrange multiplier $\gamma_h$. The approximate 
$\gamma_h$ is then designed such that the constraint $|d|=1$ 
holds. Clearly, this leads to a $\gamma_h$ which depends 
nonlinearly and implicit on the unknown $d$ and
hence leads to a rather unpractical method. 

The method we will develop in this paper differs significantly 
from the previously proposed methods.
The key observation allowing
us to deduce an energy and constraint preserving 
method is a new formulation of 
\eqref{eq:3Dwave}. Specifically, by introducing the \emph{angular momentum}:
$$
    w = d_t \times d.
$$
the wave map equation \eqref{eq:3Dwave} can be recast in the form
\begin{align}
    d_t &= d\times w, \label{eq:w1} \\
    w_t &= \Delta d\times  d \label{eq:w2}.
\end{align}
In this formulation, the constraint $|d|=1$ is inherit and 
hence there is no need for the Lagrange multiplier $\gamma$.
Constraint preserving time integration for this system is easily derived. Here, 
we will use the  first order integration:
\begin{equation}\label{integration}
    \begin{split}
        \frac{d^{m+1} - d^{m}}{\Delta t} &= w^{m+1/2}\times d^{m+1/2} \\
        \frac{w^{m+1} - w^{m}}{\Delta t} &= \Delta d^{m+1/2}\times d^{m+1/2},
    \end{split}
\end{equation}
where $d^{m+1/2}= \frac{1}{2}(d^m + d^{m+1})$ and similarly $w^{m+1/2}=\frac{1}{2}(w^m + w^{m+1})$. 
This integration method satisfies $|d^{m+1}| = |d^{m}|$. Moreover, by dotting the 
second equation with $w^{m+1/2}$ and adding the first equation dotted with $\Delta d^{m+1/2}$, 
one obtains 
\begin{equation*}
    \int_\Om |w^{m+1}|^2 + |\Grad d^{m+1}|^2~dx = \int_\Om |w^{m}|^2 + |\Grad d^{m}|^2~dx,
\end{equation*}
and hence the method also conserves the energy. 
To discretize \eqref{integration} in space, we will use a standard 
central difference approximation of the Laplace operator on 
a regular grid. 

The only potential downside of using the discretization \eqref{integration} is that it is nonlinear and 
implicit and hence requires implementing a fixed point solver. 
Moreover, this fixed point solver should be such that at least the length constraint
is conserved at every iteration.
In Section 4, we will give the details on how such a solver may be 
constructed and prove that a fixed point may be computed (up to any tolerance in energy norm)
using only $N \log N$ operations, where $N$ is the number of degrees of freedom of $d$.
In practice, finding a solution with tolerance $N^{-2}$ requires only around 
$5 - 10$ iterations depending on the regularity of the underlying solution, 
but not on $N$. Note that there is not much point in decreasing the tolerance 
beyond $N^{-2}$ as the discretization error of \eqref{integration} will then 
dominate the error. 

Our main theoretical result in this paper is that the new angular momentum method 
converges to a weak solution as discretization parameters go to zero. 
The proof of this fact will follow directly using energy arguments together 
with the observation that 
$$
\Delta d \times d = \Div (A(d)\Grad d), 
$$
for some matrix $A$. 

The remaining parts of this paper are structured as follows:
In the upcoming section, we will properly define the new  method and 
 prove some basic properties. Then, in Section 3, we will prove that the method
converges to a weak solution as discretization parameters go to zero.
In Section 4, we will provide a way 
to compute the needed fixed point through an iterative procedure and 
prove that a fixed point may be obtained using only $N \log N$ operations. 
In Section 5, the paper is concluded by a series of numerical experiments 
illuminating some of the properties of the new method.

\section{The angular momentum method}

Given a number of degrees of freedom $N$, we set $M=N^\frac{1}{n}$,
where $n=2$, $3$ is the spatial dimension,  
and assume that $M$ is an integer.
Next, we let $h = 1/M$ and set the time step $\Delta t = \kappa h$,
where $\kappa$ is some constant. The domain $\Om$
is then discretized in terms of the $N$ points
$$
x_{i,j,k} = (ih,jh, kh ), \quad i,j,k= 0,\ldots, M.
$$
To simplify notation, we introduce the multiindex $\bi\in \mathcal{I}_N :=\{0,\dots,M\}^3$ such that
we can write
$$
x_{\bi} = x_{i,j,k}.
$$
We will approximate $d$ at these points. Specifically, 
$$
d^m_{\bi} \approx d(m\Delta t , x_{\bi}).
$$
Next, let  $\mathbb{1}:=(1,0,0)$, $\mathbb{2}:=(0,1,0)$, and $\mathbb{3}:=(0,0,1)$.
Using these vectors, we then define
the forward and backward difference operators
\begin{equation*}
        D^+_j d_{\bi} = \frac{d_{\bi+ \bj} -d_{\bi}}{h}, \qquad D^-_j d_{\bi} = D^+_j d_{\bi-\bj}
\end{equation*}
respectively, for $j=1,2,3$, and $\bi\in \mathcal{I}_N$. The standard central
Laplace discretization is then defined as
\begin{equation*}
    \Delta_h d_{\bi} = \sum_{j=1}^3 D_j^+D_j^- d_{\bi}.
\end{equation*}
If we introduce the backward gradient
$\Grad_h = [D^-_1, D^-_2, D^-_3]^T$
and forward divergence $\Div_h v = D^+_1 v^{(1)} + D^+_2 v^{(2)} + D^+_3 v^{(3)}$,
we have the identity
\begin{equation*}
    \Div_h \Grad_h = \Delta_h,
\end{equation*}
which will be convenient in the upcoming analysis.

For time discretization, we will use the notation
\begin{equation*}
d^{m+1/2}:=\frac{d^m+d^{m+1}}{2},\quad D_t d^m=\frac{d^{m+1}-d^m}{\Dt}.
\end{equation*} 

To approximate the initial conditions, we shall use the operator
$$
\Pi [f]_{\bi} = \frac{1}{(h)^3}\int_{(i-0.5)h}^{(i+0.5)h}\int_{(j-0.5)h}^{(j+0.5)h}\int_{(k-0.5)h}^{(k+0.5)h}f(y)~dy,
$$
with the obvious modification if $d=2$.

We are now ready to state the new method.
\equationbox{
\begin{definition}\label{def:scheme}
Given initial data $d^0 \in H^1(\Om)$, $d^0_t \in L^2(\Om)$, let 
$$
    (d^0_{\bi}, ~w^0_{\bi}) = \left(\Pi[d^0]_{\bi}, \Pi[d_t^0 \times d^0]_{\bi}\right), \quad \forall \bi.
$$
Determine sequentially, 
$$
d^m_{\bi}, w^m_{\bi}, \qquad \forall \bi \in \mathcal{I}_N, \quad m=1, \ldots,
$$
by solving the nonlinear system
    \begin{align}
        D_t d_{\bi}^m &= d_{\bi}^{m+1/2}\times w_{\bi}^{m+1/2}, \label{num:1}\\
        D_t w_{\bi}^m &= \Delta_{h} d_{\bi}^{m+1/2}\times  d_{\bi}^{m+1/2}\label{num:2}.
    \end{align}
\end{definition}
}

We will now prove some fundamental properties of 
the new method.

\begin{lemma}\label{lem:constraint}
There exists a unique numerical solution to the method 
posed in Definition \ref{def:scheme}. Moreover, 
the length is preserved
\begin{equation}\label{const:length}
    |d_i^m| = |d_i^0| = 1,\quad \forall i, \quad m=0, \ldots, 
\end{equation}
and the energy is preserved 
\begin{equation}\label{const:energy}
    E_m = E_0, \quad  m=1, \ldots,
\end{equation}
where the energy is defined as
\begin{equation}\label{eq:defnrgy}
E_m=\frac{1}{2}\int_\Omega |\Grad_h d^m_h|^2+|w^m_h|^2\, dx.
\end{equation}
\end{lemma}
\begin{proof}
The existence of a unique solution 
will be proved through a constructive iteration 
in Section 4. The proof can be found in Corollary \ref{cor:unique}.

That the length is conserved, \eqref{const:length}, follows immediately from \eqref{num:1}. 
Indeed, dotting with $d^{m+1/2}_h$ yields
$$
    D_t d_h^m\cdot d_h^{m+1/2} = 0.
$$
Finally, to prove \eqref{const:energy}, we calculate
\begin{equation*}
    \begin{split}
            D_t E_m &= \int_\Om w^{m+1/2}\cdot D_t w^m - \Delta_h d^{m+1/2}\cdot D_t d^m~dx \\
            &= \int_\Om \left(\Delta_h d^{m+1/2}\times d^{m+1/2}\right)\cdot w^{m+1/2}~ dx \\
            &\qquad -\int_\Om \left(d^{m+1/2}\times w^{m+1/2}\right)\cdot \Delta_h d^{m+1/2}~dx=0.
    \end{split}
\end{equation*}
This concludes the proof.
\end{proof}

\section{The method converges}

To prove that the method converges, it will be convenient to extend 
the numerical solution to all of $\Om$. 
For this purpose, we shall use the piecewise constant extension:
\begin{equation}\label{def:scheme-2}
    \begin{split}
        d^m_h(x) &= d_{\bi}^m, \quad x \in E_{\bi}, \\
        w^m_h(x) &= w_{\bi}^m, \quad x \in E_{\bi},
    \end{split}
\end{equation}
where $E_{\bi}=[(i-1/2) h, (i+1/2)h)\times [(j-1/2)h,(j+1/2)h)\times [(k-1/2)h,(k+1/2)h)$, $\bi=(i,j,k)\in \mathcal{I}_N$.
Observe that the numerical method can then be written 
\begin{align}
    D_t d^{m}_h &= d^{m+1/2}_h\times w_h^{m+1/2},  \label{num:3}\\
    D_t w_h^m &=  \Delta_h d_h^{m+1/2} \times d^{m+1/2}_h,\label{num:4}
\end{align}
where $\Delta_h$ is derived in the obvious way.

Our main result in this section is the following convergence result:
\begin{theorem}\label{thm:main}
Let $\{(d_h, w_h)\}_{h>0}$ be a sequence of numerical approximations obtained using Definition \ref{def:scheme}
and \eqref{def:scheme-2}, where $\Delta t = \kappa h$ for some constant $\kappa >0$. Then, as $h \rightarrow 0$,
 $d_h \rightarrow d$ a.e and in $L^p((0,\infty)\times \Om)$ for any $p< \infty$, 
 $w_h \weakstar w$ in $L^\infty(0,T;L^2(\Om))$, where $(d,w)$ is a weak solution of the wave map equation \eqref{eq:w1}--\eqref{eq:w2}.
 By a weak solution, we mean that $(d,w)$ satisfies
 \begin{equation}\label{eq:weak}
     \begin{split}
         \int_0^\infty \int_\Om d\phi_t + (d\times w)\phi~dxdt+ \int_{\Om} d^0  \phi(0,\cdot)\, dx &= 0, \\
         \int_0^\infty \int_\Om w\psi_t + (d\times \Grad^T d):\Grad^T \psi~dxdt+\int_{\Om} w^0  \psi(0,\cdot)\, dx &=0,
     \end{split}
 \end{equation}
for all  $\phi, \psi \in C^\infty_0([0,\infty)\times \Om;\mathbb{R}^n)$, where the supscript $T$ means 
the transposed gradient matrix.
\end{theorem}

To prove this theorem, our starting point is Lemma \ref{lem:constraint}  yielding 
the $h$-independent bounds:
\begin{equation*}
    \begin{split}
        D_t d_h &\in_b L^\infty(0,\infty;L^2(\Om)), \\
        \Grad_h d_h&\in_b L^\infty(0,\infty;L^2(\Om)), \\
        w_h &\in_b L^\infty(0,\infty;L^2(\Om)),
    \end{split}
\end{equation*}
where the $\in_b$ means that the inclusion is independent of $h$.
From these bounds, we can assert the existence of functions $d$ and $w$,
and a subsequence $h_j$, such that 
\begin{equation}\label{eq:conv}
    \begin{split}
        w_h &\weakstar w \text{ in $L^\infty(0,\infty;L^2(\Om))$}, \\
        D_t d_h &\weakstar d_t \text{ in $L^\infty(0,\infty;L^2(\Om))$}, \\
        \Grad_h d_h &\weakstar \Grad d \text{ in $L^\infty(0,\infty;L^2(\Om))$}, \\
        d_h &\rightarrow d \text{ a.e and in $L^p((0,\infty)\times \Om)$ for $p < \infty$},
    \end{split}
\end{equation}
where the limit $d$ also satisfies the constraint
$$
|d(t,x)| = 1 \text{ a.e in $[0,\infty)\times \Om$}.
$$
To show that the limit pair $(d,w)$ is a weak solution of \eqref{eq:w1} -- \eqref{eq:w2}, we will need a vector identity.
It is this identity which allows us to pass to the limit without 
having higher-order bounds on $d$.
\begin{lemma}\label{lem:cp}
The following identity holds,
    \begin{equation}\label{eq:cp}
        \begin{split}
                d_h\times \Delta_h d_h 
                &= \Div_h \begin{pmatrix}
                    d_h^{(2)}\Grad_h^T d_h^{(3)} - d_h^{(3)} \Grad_h^T d_h^{(2)} \\
                    d_h^{(3)}\Grad_h^T d_h^{(1)} - d_h^{(1)} \Grad_h^T d_h^{(3)} \\
                    d_h^{(1)}\Grad_h^T d_h^{(2)} - d_h^{(2)} \Grad_h^T d_h^{(1)} 
                \end{pmatrix}.
        \end{split}
    \end{equation}
\end{lemma}
\begin{proof}
By direct calculation, we see that
\begin{equation*}
d^{(k)}_{\bi}\, \Delta_h d_{\bi}^{(\ell)} = \Div_h \left(d^{(k)}_{\bi} \Grad_h d_{\bi}^{(\ell)}\right)-\Grad_h d_{\bi}^{(\ell)}\cdot \Grad_h d_{\bi}^{(k)},
\end{equation*}
for $k,\ell\in\{1,2,3\}$.
Hence
\begin{equation*}
d^{(k)}_{\bi} \,\Delta_h d_{\bi}^{(\ell)} -d^{(\ell)}_{\bi} \,\Delta_h d_{\bi}^{(k)}=\Div_h \left(d^{(k)}_{\bi} \Grad_h d^{(\ell)}_{\bi} - d^{(\ell)}_{\bi} \Grad_h d^{(k)}_{\bi}\right)
\end{equation*}
and \eqref{eq:cp} follows.
\end{proof}

\subsection{Proof of Theorem \ref{thm:main}:}
For test functions $\varphi$, $\psi\in C^1_0([0,\infty)\times \Omega;\mathbb{R}^n)$, we denote $\varphi^m(x):=\varphi(t^m,x)$, $\psi^m(x):=\psi(t^m,x)$. Then we dot \eqref{num:3} and \eqref{num:4} with $\varphi^m, \psi^m$, integrate over $\Omega$, and sum over $m$, to discover
\begin{align*}
\Dt \sum_{m=0}^{\infty} \int_\Omega \left(D_t d^{m}_h -d^{m+1/2}_h\times w_h^{m+1/2}\right)\cdot \varphi^m \, dx &= 0, \\
\Dt \sum_{m=0}^{\infty} \int_\Omega \left(D_t w_h^m + d^{m+1/2}_h \times \Delta_h d_h^{m+1/2}\right)\cdot \psi^m \, dx  &= 0.
\end{align*}
Using Lemma \ref{lem:cp} and summation by parts, we deduce that
\begin{equation}\label{eq:wf}
\begin{split}
\Dt \sum_{m=0}^{\infty} \int_\Omega \left(- d^{m+1}_h\cdot D_t \varphi^{m} -\left(d^{m+1/2}_h\times w_h^{m+1/2}\right)\cdot \varphi^m\right)dx\quad&\\ 
 -\int_{\Om} d^0_h\cdot \varphi^0\, dx &= 0, \\
\Dt \sum_{m=0}^{\infty} \int_\Omega \left(- w_h^{m+1}\cdot  D_t \psi^{m} - \left(d^{m+1/2}_h \times \Grad_h^T d_h^{m+1/2}\right): \Grad_h^T \psi^m \right) dx\quad &\\
-\int_{\Om} w^0_h\cdot \psi^0\, dx  &= 0,
\end{split}
\end{equation}
We denote 
\begin{equation*}
    \begin{split}
        d_h(t,x) &= d_h^m(x), \quad x \in \Omega, \quad t \in (t^{m-1}, t^m],\\
        w_h(t,x) &= w_h^m(x), \quad x \in \Omega, \quad t \in (t^{m-1}, t^m],\\
	\overline{d}_h(t,x) &= d_h^{m-1/2}(x), \quad x \in \Omega, \quad t \in (t^{m-1}, t^m],\\
        \overline{w}_h(t,x) &= w_h^{m-1/2}(x), \quad x \in \Omega, \quad t \in (t^{m-1}, t^m],\\
    \end{split}
\end{equation*}
such that \eqref{eq:wf} becomes
\begin{equation*}
\begin{split}
-\int_0^\infty\int_\Omega \left( d_h\cdot D_t \varphi +\left(\overline{d}_h\times \overline{w}_h\right)\cdot \varphi\right)\, dx\,dt - \int_{\Om} d_h^0\cdot \varphi(0,\cdot)\, dx &= 0, \\
-\int_0^\infty\int_\Omega \left(w_h\cdot  D_t \psi + \left(\overline{d}_h \times \Grad_h^T \overline{d}_h\right): \Grad_h^T \psi \right) \, dx  - \int_{\Om} w_h^0\cdot \psi(0,\cdot)\, dx  &= 0.
\end{split}
\end{equation*}

Now, since $\Grad_h \psi\rightarrow \Grad \psi$ a.e and $(D_t\varphi,D_t \psi)\rightarrow (\varphi_t,\psi_t)$ a.e, 
there is no problems with applying the convergences \eqref{eq:conv} to discover that the limit $(d,w)$ satisfies \eqref{eq:weak}.
Hence, $(d,w)$ is a weak solution of \eqref{eq:w1}--\eqref{eq:w2} and the proof of Theorem \ref{thm:main} is complete.
\qed

\section{A solution may be obtained fast}
The new \emph{angular momentum} method  (Definition \ref{def:scheme}) 
is both nonlinear and implicit. Hence, in practice, finding 
a solution requires solving a fixed point problem at each time step. 
In this section, we will construct a fixed point iteration scheme 
and prove that this scheme provides 
the desired solution using only $N \log N$ operations.

To find a solution of \eqref{num:1}--\eqref{num:2}, we propose the following 
iterative scheme:
\equationbox{
\begin{definition}\label{def:fixed}
Given $h >0 $, $\Delta t = \kappa h$, and functions 
$(d^m_h, w^m_h)$ satisfying \eqref{num:1}--\eqref{num:2}, we approximate
the next time-step $(d^{m+1}_h, w^{m+1}_h)$ to a given 
tolerance $\eps > 0$ by the following procedure: Set
\begin{equation*}
    (d_h^{m, 0}, w_h^{m,0}) = (d_h^m, w_h^m),
\end{equation*}
and iteratively solve $(d_h^{m,s+1}, w_h^{m,s+1})$ satisfying
\begin{equation}\label{num:fixed}
    \begin{split}
        \frac{d^{m,s+1}_h-d^m_h}{\Delta t} &= \frac{1}{2}\left(d_h^m + d_h^{m,s+1}\right)\times \frac{1}{2}\left(w^{m}_h + w^{m,s}_h\right), \\
        \frac{w^{m,s+1}_h-w^m_h}{\Delta t} &= \frac{1}{2}\left(\Delta_h d_h^{m} + \Delta_h d_h^{m,s+1}\right)\times \frac{1}{2}\left(d_h^m + d_h^{m,s+1}\right),
    \end{split}
\end{equation}
until the following stopping criteria is met:
\begin{equation}\label{eq:stop}
    \left\|w_h^{m, s+1} - w_h^{m,s}\right\|_{L^2(\Om)} + \left\|\Grad d_h^{m,s+1} - \Grad d_h^{m,s}\right\|_{L^2(\Om)} < \eps.
\end{equation}

\end{definition}
}

Clearly, if the iteration  \eqref{num:fixed} yields a fixed point $w_h^{m,s+1} = w_h^{m,s}$, then $(d^{m+1}_h, w_h^{m+1}) = (d_h^{m,s+1},w_h^{m,s+1})$  is a solution to the nonlinear scheme
\eqref{num:1}--\eqref{num:2}. Moreover, the iteration in \eqref{num:fixed} is put
up precisely such that the length is preserved at each iteration:
$$
|d_h^{m,s}| = |d_h^m| = 1 \quad \text{in $\Om$}.
$$

Seen from the practical point of view, the remaining questions are whether the iteration 
converges or not and, if so, how many iterations that are needed to reach 
the given tolerance $\eps$. The following theorem provides an answer
to these questions and is our main result in this section.

\begin{theorem}\label{thm:it}
Given $h > 0$, $\Delta t = \kappa h$ for a sufficiently small 
$\kappa > 0$, 
and a small tolerance $\eps > 0$, 
there is a number of iterations $\bar s \in \mathbb{N}_+$, $\bar s\leq C |\log \eps|$, such that
\eqref{eq:stop} holds and the error
\begin{equation}\label{eq:error}
    \left\|w_h^{m+1} - w_h^{m,  s}\right\|_{L^2(\Om)} + \left\|\Grad d_h^{m+1} - \Grad d_h^{m, s}\right\|_{L^2(\Om)} < \eps, \quad \forall s \geq \bar s.
\end{equation}
\end{theorem}
The proof of this theorem will follow as a consequence of the 
results stated and proved in the remaining parts of this section.
\begin{remark}
    In Theorem \ref{thm:it}, we need that $\kappa$ is sufficiently small. Upon inspecting 
    the upcoming proof, one can derive that $\kappa \leq \frac{1}{50}$. However, 
    in practice, it is sufficient to have $\kappa \leq \frac{1}{2}$. This is the only 
    instance at which we need to require anything on $\kappa$.
\end{remark}

As an immediate corollary of Theorem \ref{thm:it}, we have that a desired solution may be 
computed in $N\log N$ operations:
\begin{corollary}\label{cor:it}
For a given tolerance $\eps = N^{-\alpha}$, the functions $(d_h^{m,\bar s}, w_h^{m, \bar s})$
in Theorem \ref{thm:it} may be computed using only $O(N\log N)$ operations.
\end{corollary}
\begin{proof}
Since each iteration requires $N$ operations and we need $O(|\log \eps|)$
iterations, we get a total of $O(N |\log \eps|)$ iterations 
and the proof follows by inserting $\eps = N^{-\alpha}$.
\end{proof}

Another consequence of Theorem \ref{thm:it} is that the energy 
at the stopping time $\bar s$ is almost conserved:
\begin{corollary}\label{cor:energy}
Under the conditions of Theorem \ref{thm:it},
\begin{equation*}
    E_m^{\bar s}:=\frac{1}{2}\left(\left\|w_h^{m,\bar s}\right\|_{L^2(\Om)}^2 + \left\|\Grad_h d_h^{m,\bar s}\right\|_{L^2(\Om)}^2\right)
    = E_0 + \mathcal{O}\left( \eps^2\right).
\end{equation*}
\end{corollary}
\begin{proof}
By multiplying the first equation in \eqref{num:fixed} with $-\frac{1}{2}\Delta_h (d_h^{m,s+1} + d_h^{m})$,
the second equation with $\frac{1}{2}(w_h^{m,s} + w_h^m)$, and integrating by parts, we obtain that
\begin{equation*}
    \begin{split}
        &\frac{1}{2}\left(\left\|w_h^{m,s+1}\right\|_{L^2(\Om)}^2 + \left\|\Grad_h d_h^{m,s+1}\right\|_{L^2(\Om)}^2\right) \\
        &\qquad = E_m + \frac{1}{2}\int_\Om (w_h^{m,s+1} - w_h^m)(w_h^{m,s+1}-w_h^{m,s})~dx \\
        &\qquad = E_0 + \frac{1}{2}\int_\Om (w_h^{m,s+1} - w_h^m)(w_h^{m,s+1}-w_h^{m,s})~dx,
    \end{split}
\end{equation*}
where the last inequality follows from Lemma \ref{lem:constraint}. 

Finally, we 
assume that $s>\bar s$ such that both \eqref{eq:stop} and \eqref{eq:error} holds. 
The Cauchy-Schwarts inequality then provides the estimate
\begin{equation*}
    \begin{split}
        &\frac{1}{2}\int_\Om (w_h^{m,s+1} - w_h^m)(w_h^{m,s+1}-w_h^{m,s})~dx \\
        & \leq \frac{1}{2}\|w_h^{m,s+1} - w_h^m\|_{L^2(\Om)}\|w_h^{m,s+1}-w_h^{m,s}\|_{L^2(\Om)} \leq \frac{\eps^2}{2},
    \end{split}
\end{equation*}
which brings the proof to an end.
\end{proof}

\begin{remark}
In practice, the $(d_h^m,w_h^m)$ appearing in the fixed point scheme \eqref{num:fixed}
would be the approximation coming from the previous time-step. In this case Corollary \ref{cor:energy}
tell us that the "error" will be summed and thus
\begin{equation*}
    E_m^{\bar s} = E_0 + \mathcal{O}\left(m \eps^2\right).
\end{equation*}
\end{remark}

\subsection{The fixed point map $F_m$}
To prove Theorem \ref{thm:it}, it will be convenient 
to write the fixed point iteration in terms of a map. 
To define this map, we first notice that \eqref{num:1}, 
can be rewritten as
\begin{equation}\label{eq:weissnoed}
d^{m+1}_{\bi}=V(w^{m+1/2}_{\bi})\, d^m_{\bi}
\end{equation}
where $V=V(w)$ is the following matrix
{
\begin{align}\label{eq:V}
V(w)&=
\frac{1}{1+\frac{\Dt^2}{4}|w|^2}\left(\left(1-\frac{\Dt^2}{4}|w|^2\right) \mathbb{1}+ \frac{\Dt^2}{2}(w \otimes w)+\Dt Q(w) \right),
\end{align}}
and $Q(w)$ is defined as
\begin{equation*}
Q(w)=\begin{pmatrix}
0 & w^{(3)} & -w^{(2)}\\
-w^{(3)} & 0 & w^{(1)}\\
w^{(2)} & -w^{(1)} & 0 \end{pmatrix}.
\end{equation*}
In particular, $Q(\cdot)$ is such that
\begin{equation*}
Q(w) v=v\times w
\end{equation*}
for any vector $v\in\mathbb{R}^3$. Note that $V$ is an orthogonal matrix, and therefore, independently of $w$,
\begin{equation*}
|V(w) v|^2=|v|^2
\end{equation*}
for any $v\in \mathbb{R}^3$. 

To prove the theorem, we will demonstrate that $w^{m+1}_h$ is the fixed point of a contractive mapping $F_m$ which is defined as follows:
\begin{definition}[The mapping $F_m$]\label{def:Fm}
For a piecewise constant function $u_h$ on $\Omega$,
\begin{equation}\label{eq:pwc}
u_h(x) = u_{\bi}, \quad x \in E_{\bi},\, \bi\in \mathcal{I}_N,
\end{equation}
for some $\{u_{\bi}\}_{\bi\in\mathcal{I}_N}$, we define the piecewise constant function $v_h:=F_m(u_h)$ by
\begin{equation*}
v_h(x) = v_{\bi}, \quad x \in E_{\bi},\, \bi\in \mathcal{I}_N,
\end{equation*}
where $v_{\bi}$, $\bi\in \mathcal{I}_N$ is given by
\begin{align}\label{eq:Fm}
\begin{split}
v_{\bi}&=w^m_{\bi}+\Dt\bigl[\Delta_h \left( V(\overline{u}_{\bi}) d^m_{\bi}\right)\bigr]\times \left(V(\overline{u}_{\bi}) d^m_{\bi}\right)\\
\overline{u}_{\bi}&=\frac{w^m_{\bi}+u_{\bi}}{2},\quad \bi\in \mathcal{I}_N.
\end{split}
\end{align}
\end{definition}
A fixed point $v_h=v^*_h$ of $F_m$, will be a solution to \eqref{num:2} and $d^*_h$ defined as a piecewise constant interpolation of $d^*_{\bi}=V(\overline{v}^*_{\bi})\, d^m_{\bi}$, $\bi\in \mathcal{I}_N$, will be a solution to \eqref{num:1}.

\subsection{The map $F_m$ is a contraction}
We  now proceed by proving that the mapping $F_m$
is  a contraction.

\begin{lemma}\label{lem:contraction}
The mapping $F_m$ defined by \eqref{eq:Fm} is a contraction in the  $L^2(\Omega)$-norm if $\Dt\leq \kappa h$ for a constant $\kappa$ sufficiently small, that is,
\begin{equation*}
\|F_m(u_{1,h})-F_m(u_{2,h})\|_{L^2(\Omega)}\leq q \|u_{1,h}-u_{2,h}\|_{L^2(\Omega)}
\end{equation*}
for some $q<1$ for any two piecewise constant functions $u_{1,h}$, $u_{2,h}$ on $\Omega$ defined as in \eqref{eq:pwc}. In particular, by Banach's fixed point theorem, this implies that the mapping $F_m$ has a unique fixed point.
\end{lemma}

\begin{proof}
    For the ease of notation, we will omit the indices $\bi$, $m$ and $h$ and write $w$, $d$, $F$, $u_1$, $u_2$ for $w^m_h$, $d^m_h$, $F_m$, $u_{1,h}$, $u_{2,h}$, respectively. Moreover, we denote $y_1:=F(u_1)$ and $y_2:=F(u_2)$ and $\overline{u}_j:=(w+u_j)/2$, $j=1,2$, such that
    \begin{equation*}
    y_j=w+\Dt \Div_h\bigl[\nabla_h \left(V\left(\overline{u}_j\right) d\right)\times V\left(\overline{u}_j\right) d\bigr],\quad j=1,2.
    \end{equation*}
    Then, using the inverse inequality,
    \begin{align}\label{eq:Y}
    \begin{split}
    \|y_1-y_2\|&=\Dt\bigl\|\Div_h\bigl[\nabla_h \left(V(\oxe) d\right)\times V(\oxe)-\nabla_h \left(V(\oxz) d\right)\times V(\oxz) d\bigr]\bigr\|_{L^2(\Om)}\\
    &\leq C \frac{\Dt}{h} \bigl\|\nabla_h \left(V(\oxe) d\right)\times V(\oxe)-\nabla_h \left(V(\oxz) d\right)\times V(\oxz) d\bigr\|_{L^2(\Om)}\\
    &\leq C\frac{\Dt}{h}\Bigl( \bigl\|\nabla_h \left([V(\oxe)-V(\oxz)] d\right)\times V(\oxe)d\bigr\|_{L^2(\Om)}\\
    &\hphantom{\leq C\frac{\Dt}{h}\Bigl(}+\bigl\|\nabla_h \left(V(\oxz) d\right)\times [V(\oxe)-V(\oxz)] d\bigr\|_{L^2(\Om)}\Bigr)\\
    &\leq C\frac{\Dt}{h^2} \bigl\|[V(\oxe)-V(\oxz)] d\bigr\|_{L^2(\Om)}
    \end{split}
    \end{align}
    using that $|V(\overline{u}_j)d|\leq 1$ for the last inequality. We split $\|[V(\oxe)-V(\oxz)] d\|_{L^2(\Om)}$ using \eqref{eq:V},
    \begin{align*}
    &\bigl\|[V(\oxe)-V(\oxz)] d\bigr\|_{L^2(\Om)}\\
    &\qquad \leq \biggl\|\biggl[\frac{1-\frac{\Dt^2}{4}|\oxe|^2}{1+\frac{\Dt^2}{4}|\oxe|^2}
    -\frac{1-\frac{\Dt^2}{4}|\oxz|^2}{1+\frac{\Dt^2}{4}|\oxz|^2}\biggr] d\biggr\|_{L^2(\Om)}\\
    &\qquad \hphantom{\leq}+\biggl\|\biggl[ \frac{\Dt^2}{2+\frac{\Dt^2}{2}|\oxe|^2}(\oxe \otimes \oxe)-\frac{\Dt^2}{2+\frac{\Dt^2}{2}|\oxz|^2}(\oxz \otimes \oxz)\biggr] d\biggr\|_{L^2(\Om)}\\
    &\qquad \hphantom{\leq}+\biggl\|\biggl[\frac{\Dt}{1+\frac{\Dt^2}{4}|\oxe|^2} Q(\oxe)-\frac{\Dt}{1+\frac{\Dt^2}{4}|\oxz|^2} Q(\oxz)\biggr] d\biggr\|_{L^2(\Om)}\\
    &\qquad =:I+II+III.
    \end{align*}
    
    For the $I$ term, we apply the Cauchy-Schwartz inequality to discover that
    \begin{align}\label{eq:A}
    \begin{split}
    I&=\biggl\|\frac{\frac{\Dt^2}{2}\left(|\oxe|^2-|\oxz|^2\right)}{\left(1+\frac{\Dt^2}{4}|\oxe|^2\right)\left(1+\frac{\Dt^2}{4}|\oxz|^2\right)} d\biggr\|_{L^2(\Om)}\\
    &\leq \Dt\biggl\|\frac{1+\frac{\Dt^2}{4}|\oxe|^2+\frac{\Dt^2}{4}|\oxz|^2}{\left(1+\frac{\Dt^2}{4}|\oxe|^2\right)\left(1+\frac{\Dt^2}{4}|\oxz|^2\right)} |\oxe-\oxz| d\biggr\|_{L^2(\Om)}\\
    &\leq \Dt \|u_1-u_2\|_{L^2(\Om)}.
    \end{split}
    \end{align}
    where we have used  $|d| = 1$ to conclude the last inequality.
    
    To bound the $II$ term, we first see that
    \begin{align*}
    II&=\frac{\Dt^2}{2}\biggl\| \frac{\bigl(1+\frac{\Dt^2}{4}|\oxz|^2\bigr)(\oxe \otimes \oxe)-\bigl(1+\frac{\Dt^2}{4}|\oxe|^2\bigr)(\oxz \otimes \oxz)}{\bigl(1+\frac{\Dt^2}{4}|\oxe|^2\bigr)\bigl(1+\frac{\Dt^2}{4}|\oxz|^2\bigr)} \,d\biggr\|_{L^2(\Om)}\\
    &=\frac{\Dt^2}{2}\Biggl(\int\!\! \alpha\sum_{i=1}^3\biggl(      \sum_{j=1}^3 \biggl[\Bigl(1+\frac{\Dt^2}{4}|\oxz|^2\Bigr)\oxe^{(i)} \oxe^{(j)} \\
    & \hphantom{ \frac{\Dt^2}{2}\Biggl(\int\!\! \alpha\sum_{i=1}^3\biggl(      \sum_{j=1}^3 \biggl[\Bigl(} - \Bigl(1+\frac{\Dt^2}{4}|\oxe|^2\Bigr)\oxz^{(i)} \oxz^{(j)}\biggr] d^{(j)}\biggr)^2 dx\Biggr)^{\frac{1}{2}}
    \end{align*}
    where 
    $$
    \alpha=\frac{1}{\bigl(1+\frac{\Dt^2}{4}|\oxe|^2\bigr)\bigl(1+\frac{\Dt^2}{4}|\oxz|^2\bigr)}.
    $$
    Since $|d^{(j)}|\leq 1$, $j=1,2,3$,
    \begin{align*}
    II&\leq \frac{\Dt^2}{2}\Biggl(\int \alpha\sum_{i=1}^3\biggl( \sum_{j=1}^3\bigg| \Bigl(1+\frac{\Dt^2}{4}|\oxz|^2\Bigr)\oxe^{(i)} \oxe^{(j)}\\
    & \hphantom{\leq \frac{\Dt^2}{2}\Biggl(\int \alpha\sum_{i=1}^3\biggl(      \sum_{j=1}^3\bigg| \Bigl(} - \Bigl(1+\frac{\Dt^2}{4}|\oxe|^2\Bigr)\oxz^{(i)} \oxz^{(j)} \bigg|\biggr)^2\, dx\Biggr)^{\frac{1}{2}}\\
    & \leq \frac{3\Dt^2}{2}\sum_{i,j=1}^3 \biggl\|\alpha\biggl(\Bigl(1+\frac{\Dt^2}{4}|\oxz|^2\Bigr)\oxe^{(i)} \oxe^{(j)} - \Bigl(1+\frac{\Dt^2}{4}|\oxe|^2\Bigr)\oxz^{(i)} \oxz^{(j)}\biggr)\biggr\|_{L^2(\Om)}.
    \end{align*}
    We consider one of the summands:
    \begin{align*}
    &\biggl\|\alpha\biggl(\Bigl(1+\frac{\Dt^2}{4}|\oxz|^2\Bigr)\oxe^{(i)} \oxe^{(j)} - \Bigl(1+\frac{\Dt^2}{4}|\oxe|^2\Bigr)\oxz^{(i)} \oxz^{(j)}\biggr)\biggr\|_{L^2(\Om)}\\
    &\leq \Bigl\|\alpha (\oxe^{(i)} \oxe^{(j)}-\oxz^{(i)} \oxz^{(j)})\Bigr\|_{L^2(\Om)}+\frac{\Dt^2}{4}\Bigl\|\alpha\left(|\oxz|^2\oxe^{(i)} \oxe^{(j)}-|\oxe|^2\oxz^{(i)} \oxz^{(j)}\right)\Bigr\|_{L^2(\Om)}\\
    &=:II_1+II_2.
    \end{align*}
    By adding and subtracting, and applying the Cauchy-Schwartz inequality, we deduce the following
    bound for the first term,
    \begin{align}
    II_1&=\frac{1}{\Dt} \bigl\|\alpha (\Dt\,\oxe^{(i)} (u_1^{(j)}-u_2^{(j)})+(u_1^{(i)}-u_2^{(i)}) \Dt\,\oxz^{(j)})\bigr\|_{L^2(\Om)} \nonumber\\
    &\leq \frac{1}{\Dt}\Bigl\{ \bigl\|\alpha \Dt\oxe^{(i)} (u_1^{(j)}-u_2^{(j)})\bigr\|_{L^2(\Om)}+\bigl\|\alpha(u_1^{(i)}-u_2^{(i)}) \Dt\oxz^{(j)}\bigr\|_{L^2(\Om)}\Bigr\}\nonumber\\
    &\leq \frac{1}{2\Dt}\Bigl\{ \bigl\|\alpha (1+\Dt^2(\oxe^{(i)})^2) (u_1^{(j)}-u_2^{(j)})\bigr\|_{L^2(\Om)} \nonumber\\
    &\hphantom{\leq \frac{1}{2\Dt}\Bigl\{ \bigl\|} +\bigl\|\alpha(u_1^{(i)}-u_2^{(i)}) (1+\Dt^2(\oxz^{(j)})^2)\bigr\|_{L^2(\Om)}\Bigr\} \nonumber\\
    &\leq\frac{4}{\Dt} \|u_1-u_2\|_{L^2(\Om)},\label{eq:ii1}
    \end{align}
    where the last inequality follows by inserting the definition of $\alpha$.
    
    Term $II_2$ may be written as
    \begin{equation*}
    II_2=\frac{\Dt^2}{4}\biggl\|\alpha\sum_{k=1}^3\left( (\oxz^{(k)})^2\oxe^{(i)} \oxe^{(j)}-(\oxe^{(k)})^2\oxz^{(i)} \oxz^{(j)}\right)\biggr\|_{L^2(\Om)}.
    \end{equation*}
    We consider one of the terms in the sum. Note that if $i=j$, the term where $i=j=k$ cancels, hence we can assume without loss of generality that $i\neq k$.
    By adding and subtracting, we rewrite one of the terms in $II_2$ as follows
    \begin{equation*}
        \begin{split}
    &(\oxz^{(k)})^2\oxe^{(i)} \oxe^{(j)}-(\oxe^{(k)})^2\oxz^{(i)} \oxz^{(j)} \\
    &\qquad =\oxe^{(i)}\oxe^{(j)}\oxz^{(k)}(u_2^{(k)}-u_1^{(k)})+\oxe^{(k)}\oxz^{(k)}\oxe^{(j)}(u_1^{(i)}-u_2^{(i)}) \\
    &\qquad \qquad +\oxe^{(k)}\oxz^{(k)}\oxz^{(i)}(u_1^{(j)}-u_2^{(j)})+\oxz^{(j)}\oxz^{(i)}\oxe^{(k)}(u_2^{(k)}-u_1^{(k)}).            
        \end{split}
    \end{equation*}
    Next, we apply  Young's inequality to the previous identity giving
    \begin{align*}
    &\big|(\oxz^{(k)})^2\oxe^{(i)} \oxe^{(j)}-(\oxe^{(k)})^2\oxz^{(i)} \oxz^{(j)}\big|\\
    &\quad\leq \frac{1}{\Dt}\left(|\oxe|^2+|\oxz|^2+\frac{\Dt^2}{4}|\oxe|^2|\oxz|^2\right)\\
    &\qquad \qquad\qquad\times\left(|u_1^{(i)}-u_2^{(i)}|+|u_1^{(j)}-u_2^{(j)}|+|u_1^{(k)}-u_2^{(k)}|\right) \\
    &\quad = \frac{4}{(\Delta t)^3}\left(\frac{1}{\alpha}-1\right)
    \times\left(|u_1^{(i)}-u_2^{(i)}|+|u_1^{(j)}-u_2^{(j)}|+|u_1^{(k)}-u_2^{(k)}|\right). 
    \end{align*}
    As a consequence, we conclude that
    \begin{equation}\label{eq:ii2}
        \begin{split}
            II_2&=\frac{\Dt^2}{4}\biggl\|\alpha\sum_{k=1}^3\left( (\oxz^{(k)})^2\oxe^{(i)} \oxe^{(j)}-(\oxe^{(k)})^2\oxz^{(i)} \oxz^{(j)}\right)\biggr\|_{L^2(\Om)} \\
            &\leq \frac{1}{\Dt}\|\oxe - \oxz\|_{L^2(\Om)}.
        \end{split}
    \end{equation}
    From \eqref{eq:ii1} and \eqref{eq:ii2}, we have that
    \begin{equation}\label{eq:B}
        II \leq 25\Dt\|\oxe - \oxz\|_{L^2(\Om)}.
    \end{equation}
    
    The final term $III$ can be bounded as follows
    \begin{align}\label{eq:D}
    \begin{split}
    III&=\Dt\biggl\|\frac{\left(1+\frac{\Dt^2}{4}|\oxz|^2\right) Q(\oxe)-\left(1+\frac{\Dt^2}{4}|\oxe|^2\right) Q(\oxz)}{\left(1+\frac{\Dt^2}{4}|\oxe|^2\right)\left(1+\frac{\Dt^2}{4}|\oxz|^2\right)}\,  d\biggr\|_{L^2(\Om)}\\
    &=\Dt\,\Biggl\|\frac{d\times\Bigl[\bigl(1+\frac{\Dt^2}{4}|\oxz|^2\bigr)\oxe-\bigl(1+\frac{\Dt^2}{4}|\oxe|^2\bigr)\oxz\Bigr]}{\left(1+\frac{\Dt^2}{4}|\oxe|^2\right)\left(1+\frac{\Dt^2}{4}|\oxz|^2\right)}\Biggr\|_{L^2(\Om)}\\
    &\leq  \Dt\,\Biggl\|\frac{u_1-u_2+\frac{\Dt^2}{4}\left(|\oxz|^2\oxe-|\oxe|^2\oxz\right)}{\left(1+\frac{\Dt^2}{4}|\oxe|^2\right)\left(1+\frac{\Dt^2}{4}|\oxz|^2\right)}\Biggr\|_{L^2(\Om)}\\
    &\leq \Dt\,\Biggl\|\frac{\bigl[1+\frac{\Dt^2}{8}\left(|\oxe|^2+|\oxz|^2\right)](u_1-u_2)}{\left(1+\frac{\Dt^2}{4}|\oxe|^2\right)\left(1+\frac{\Dt^2}{4}|\oxz|^2\right)} \\
    &\qquad \qquad\qquad  +\frac{\frac{\Dt^2}{8}\left(|\oxe+\oxz|^2\,|\overline u_1-\overline u_2|\right)}{\left(1+\frac{\Dt^2}{4}|\oxe|^2\right)\left(1+\frac{\Dt^2}{4}|\oxz|^2\right)}\Biggr\|_{L^2(\Om)}\\
    &\leq 2 \Dt \,\|u_1-u_2\|_{L^2(\Om)}.
    \end{split}
    \end{align}
    Summing up \eqref{eq:Y}, \eqref{eq:A}, \eqref{eq:B} and \eqref{eq:D}, we find
    \begin{equation*}
    \|y_1-y_2\|_{L^2(\Om)}\leq C\frac{\Dt^2}{h^2}\|\oxe-\oxz\|_{L^2(\Om)},
    \end{equation*}
    and hence the map $F_m$ is a contraction as long as $\Dt\leq \kappa h$ for a constant $\kappa>0$ small enough.
    This concludes the proof.
\end{proof}

The previous lemma immediately provides the existence of a unique solution to \eqref{num:1}--\eqref{num:2}.
\begin{corollary}\label{cor:unique}
Given a previous time-step $(d_h^m, w_h^m)$, there exists a unique numerical solution $(d_h^{m+1}, w_h^{m+1})$
to the numerical method given in Definition \ref{def:scheme}.
\end{corollary}

\subsection{Proof of Theorem \ref{thm:it}}

Using the previous lemma, we can now prove that 
the fixed point iteration in Definition \ref{def:fixed}
will converge to the correct solution. 

Theorem \ref{thm:it}  is an immediate consequence of the following
lemma.

\begin{lemma}\label{lem:itt}
Given any $\epsilon_0>0$, 
there is a number of iterations $\bar s \in \mathbb{N}_+$ in Definition \ref{def:fixed} with $\bar s \leq C|\log {\epsilon}_0|$ 
such that \eqref{eq:stop} holds with $\epsilon=\epsilon_0$ and
\begin{equation}\label{eq:iterr}
\|w_h^{m+1}-w_h^{m+1,\bar s}\|_{L^2(\Omega)}+\|\Grad d_h^{m+1}-\Grad d^{m+1, \bar s}_h\|_{L^2(\Omega)}\leq \epsilon_0.
\end{equation}
\end{lemma}

\begin{proof}
We again omit writing the indices $h$ and $\bi$ and denote $w^{m+1,0}:=w^m$, $w^{m+1,s}:=F_m(w^{m+1,s-1})$ for $s\geq 1$. Now, since $F_m$ is a contraction with `Lipschitz' constant $q<1$ and $F_m(w^{m+1})=w^{m+1}$,
\begin{align*}
\|w^{m+1}-w^{m+1,s}\|_{L^2(\Om)}&=\|F(w^{m+1})-F(w^{m+1,s-1})\|_{L^2(\Om)} \\
&\leq q\|w^{m+1}-w^{m+1,s-1}\|_{L^2(\Om)} \\
&\leq q^s \|w^{m+1}-w^{m}\|_{L^2(\Om)}.
\end{align*}
Thus, it follows from the energy estimate, 
\begin{equation}\label{eq:kenplan}
\|w^{m+1}-w^{m+1,s}\|_{L^2}\leq 2 q^{s} E_0.
\end{equation}
Moreover, we note that it follows by the inverse inequality, \eqref{eq:weissnoed}, \eqref{eq:V} and \eqref{eq:A}, \eqref{eq:B} and \eqref{eq:D},
\begin{align}\label{eq:toll2}
&\|\nabla_hd^{m+1, s}_h-\nabla_h d^{m+1}_h\|_{L^2(\Om)} \notag\\
&\qquad =\bigl\|\nabla_h \bigl[V((w^{m}_h+w_h^{m+1,s})/2)-V(w_h^{m+1/2})\bigr]d_h^{m}\bigr\|_{L^2(\Om)}\notag\\
&\qquad \leq \frac{C}{h}\Bigl\|\left(V((w^{m}_h+w_h^{m+1,s})/2)-V(w_h^{m+1/2})\right)d_h^{m}\Bigr\|_{L^2(\Om)}\notag\\
&\qquad \leq \frac{C\Dt}{h} \|w^{m+1,s}_h-w^{m+1}_h\|_{L^2(\Om)}
\leq 2\, C  q^{s} E_0,
\end{align}
where the last inequality follows from the {\sc cfl} condition and \eqref{eq:kenplan}. Hence, by using the triangle inequality,
\begin{align*}
&\|w^{m+1,s+1}-w^{m+1,s}\|_{L^2(\Om)}\\
&\qquad\leq \|w^{m+1,s+1}-w^{m+1}\|_{L^2(\Om)}+ \|w^{m+1}-w^{m+1,s}\|_{L^2(\Om)}
 \leq 4 \,q^{s} E_0,
\end{align*}
and therefore also
\begin{equation*}
\|\nabla_h d^{m+1,s+1}-\nabla_h d^{m+1,s}\|_{L^2(\Om)} \leq 4\, C\,q^{s} E_0,
\end{equation*}
which implies that the fixed point iteration converges. That is, 
the stopping criteria \eqref{eq:stop} is met once $s$ is high enough to satisfy
\begin{equation*}
    4 (C+1)\,q^{s} E_0 < \eps_0 \quad \Rightarrow \quad s > \frac{\log \left(\frac{4(C+1)E_0}{\eps_0}\right)}{\log \left(\frac{1}{q}\right)}.
\end{equation*}
From \eqref{eq:kenplan} and \eqref{eq:toll2}, it is clear that this $s$ also satisfies
$$
\left\|w^{m+1}-w^{m+1,s}\right\|_{L^2(\Om)}+\left\|\nabla_hd^{m+1, s}_h-\nabla_h d^{m+1}_h\right\|_{L^2(\Om)}< \eps_0.
$$
This concludes the proof.
\end{proof}

\section{Numerical results}

In this final section, we shall report on some numerical experiments 
with the new angular momentum method. 
We shall consider two cases. In the first case, we will explore the rate 
of convergence of the method. In the second case, we will check whether 
the method predicts blow-up of the gradient for initial data where this is 
known to be the case. 

\subsection{Convergence test}
It is a non-trivial task to find analytical solutions of the wave map equation \eqref{eq:3Dwave} in $3D$. In $2D$ however, the 
dynamics of the wave map equation may be totally described by the linear wave equation. Specifically, 
upon introducing an angle $\vt(t,x)$ and writing 
\begin{equation*}
    d(t,x) = \begin{pmatrix}
        \cos \vt \\ \sin \vt,
    \end{pmatrix},
\end{equation*}
one easily derives that $\vt$ evolves according to the linear wave equation
\begin{equation}\label{eq:lw}
    \vt_{tt} - \Delta \vt = 0.
\end{equation}
Hence, in the $2D$ case, we can compute analytical solutions using d'Alembert's formula. 
\begin{figure}[ht] %
  \centering
  \begin{tabular}{lr}
    \includegraphics[width=0.8\textwidth]{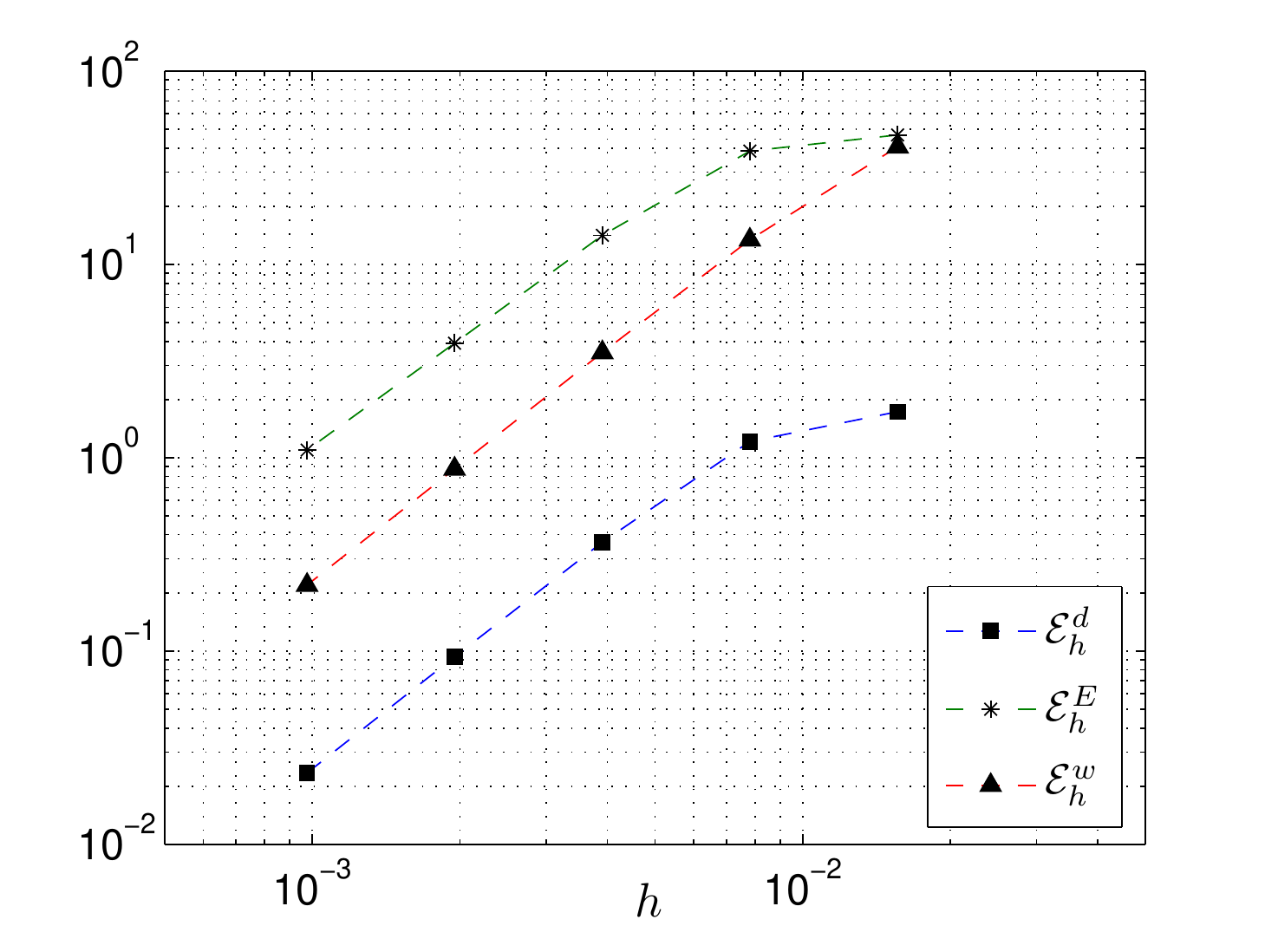}
  \end{tabular}
  \caption{The errors $\mathcal{E}^d_h$, $\mathcal{E}^w_h$ and $\mathcal{E}^E_h$ for the approximations to \eqref{eq:3Dwave} for a solution of the form \eqref{eq:dalembert} at time $T=20$ for $h=2^{-j}$, $j=6,\dots, 10$.}
  \label{fig:conv}
\end{figure}
In particular, \eqref{eq:lw} has solutions of the form
\begin{align}\label{eq:dalembert}
\begin{split}
&\vt(t,x,y)\\
&\quad =\sum_{j=-J}^J \bigl\{a_j^+ \sin(2\pi j (\sqrt{2} t +(x+y)))+ a_j^- \sin(2\pi j (\sqrt{2} t -(x+y)))\\
&\quad\hphantom{=\sum_{j=-J}^J \bigl\{}+ b_j^+ \cos(2\pi j (\sqrt{2} t +(x+y)))
+ b_j^- \cos(2\pi j (\sqrt{2} t -(x+y)))\\
&\quad \hphantom{=\sum_{j=-J}^J \bigl\{}+ c_j^+ \sin(2\pi j (\sqrt{2} t +(x-y)))+ c_j^- \sin(2\pi j (\sqrt{2} t -(x-y)))\\
&\quad\hphantom{=\sum_{j=-J}^J \bigl\{}+d_j^+ \cos(2\pi j (\sqrt{2} t +(x-y)))+ d_j^- \cos(2\pi j (\sqrt{2} t -(x-y)))\bigr\},
\end{split}
\end{align}
for $J\geq 0$. 
\begin{figure}[ht] %
  \centering
  \begin{tabular}{lr}
    \includegraphics[width=0.5\textwidth]{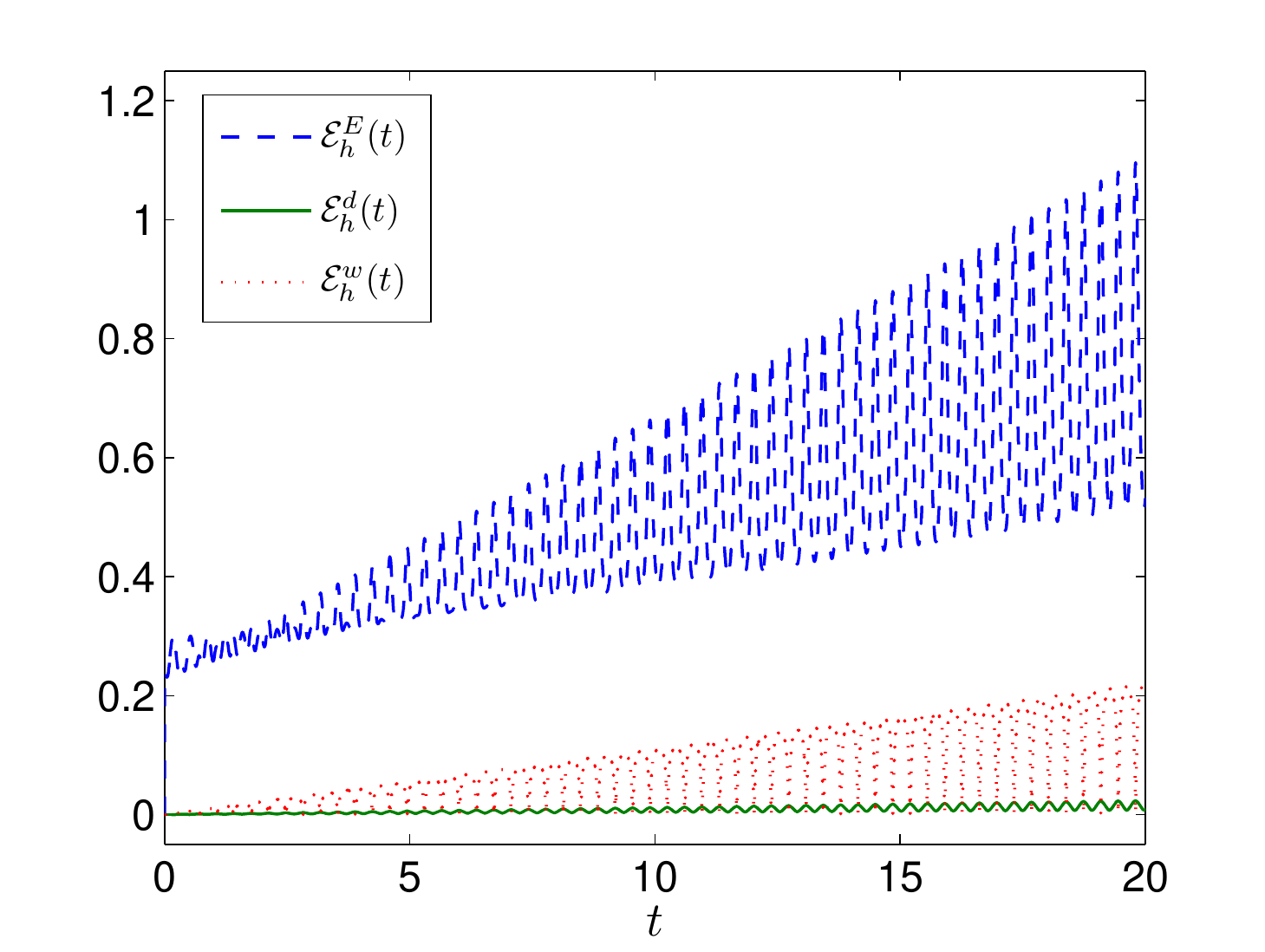}
    \includegraphics[width=0.5\textwidth]{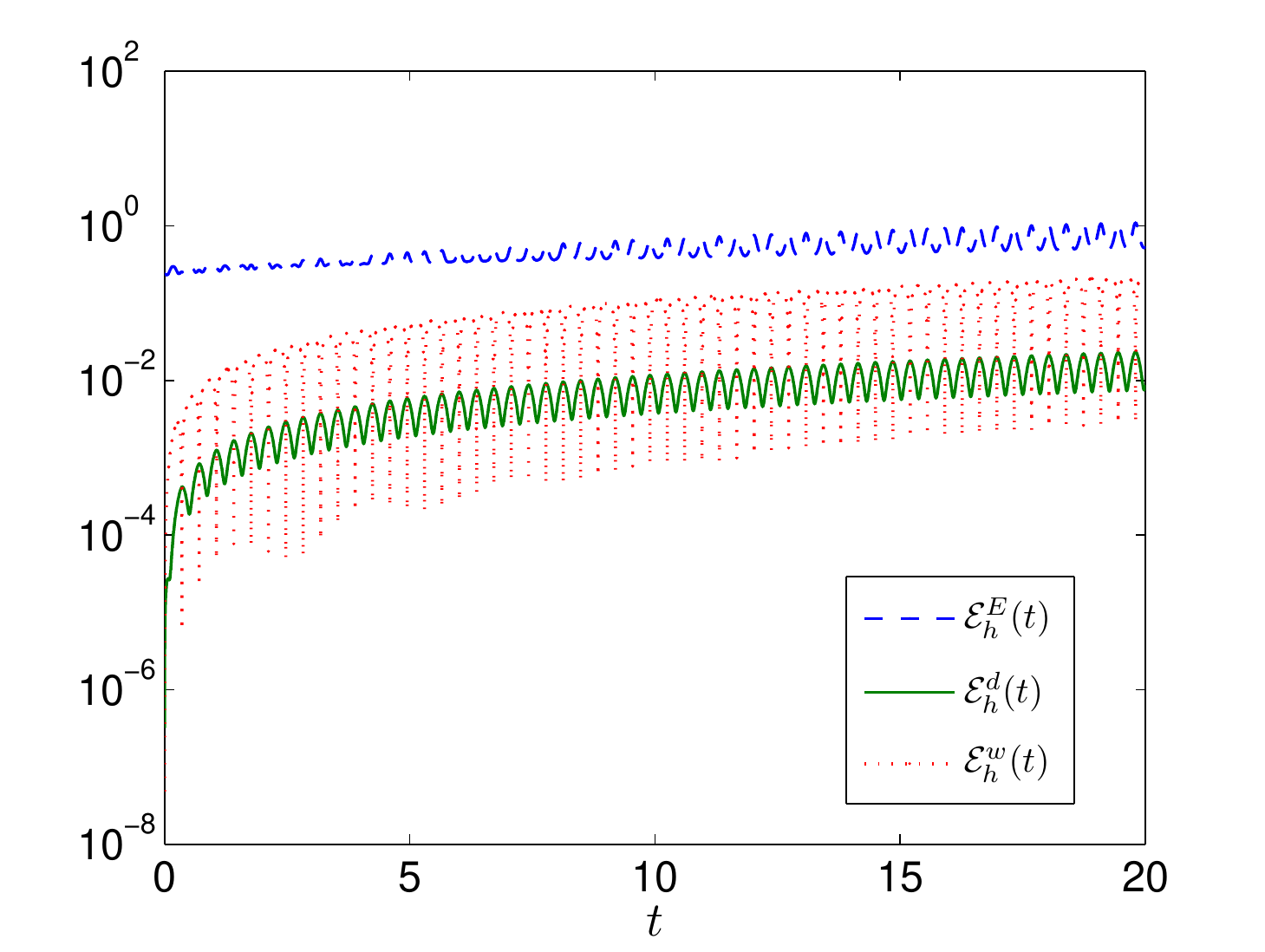}
  \end{tabular}
  \caption{The evolution of the errors $\mathcal{E}^d_h(t)$, $\mathcal{E}^w_h(t)$ and $\mathcal{E}^E_h(t)$ versus time for the approximations to \eqref{eq:3Dwave} for a solution of the form \eqref{eq:dalembert} for $h=2^{-10}$. Left: Real error, right: Error in log-scale.}
  \label{fig:errs}
\end{figure}
In Figure \ref{fig:conv}, we have plotted the errors between the approximations $d_h$ and $d$, $w_h$ and $w$ respectively, in the $L^2$-norm and in the energy norm, that is
\begin{align*}
\mathcal{E}^d_h&:=\sup_m \| d(t^m,\cdot)-d^m_h\|_{L^2(\Omega)}\\
\mathcal{E}^w_h&:=\sup_m \| w(t^m,\cdot)-w^m_h\|_{L^2(\Omega)}\\
\mathcal{E}^E_h&:=\sup_m \sqrt{\| \nabla d(t^m,\cdot)-\Grad_h d^m_h\|_{L^2(\Omega)}^2+\| d_t(t^m,\cdot)-D_t d^m_h\|_{L^2(\Omega)}^2}
\end{align*}
for $a_1^+=a_1^-=1/4$, $a_2^+=a_2^-=1/10$, $b_1^+=-b_1^-=-2$, $b_2^+=-b_2^-=1/100$, $T=20$, $\Dt=0.5 h$, and $\Omega=\mathbb{T}^2$. 
Moreover,  $h_j=2^{-j}$, $j=6,\dots,10 $ for tolerance $\epsilon_0=h^2$. We observe a rate of convergence of almost $2$ for $\mathcal{E}^d_h$ and $\mathcal{E}^w_h$ and about $1.8$ for $\mathcal{E}^E_h$ (Table \ref{tab:errs}). Other choices of $\epsilon_0$ such as $h^{3/2}$ or $h^3$ gave similar results. 
\begin{table}[h]
  \centering
  \begin{tabular}[h]{|c|c|c|c|}
    \hline
    h & $\mathcal{E}_h^d$ & $\mathcal{E}_h^E$ & $\mathcal{E}_h^w$ \\
    \hline
   
    $2^{-6}$ & 1.731 & 46.78   & 40.58\\
    $2^{-7}$  & 1.213 & 38.64 & 13.42\\
    $2^{-8}$ & 0.366  & 14.15  & 3.499\\
    $2^{-9}$  & 0.093  & 3.915  & 0.877\\
    $2^{-10}$ &0.023  &  1.096  & 0.219\\
	\hline
	Rate & $1.56$ & $1.35$ & $1.88$ \\
   
    \hline   
  \end{tabular}
   \vspace{1em}
  \caption{Errors for different mesh resolutions for \eqref{eq:3Dwave}, \eqref{eq:dalembert} at time  $T=20$, $\Dt =0.5 h$ and average rate for grid sizes $2^{-6}$ to $2^{-10}$.} 
  \label{tab:errs}
\end{table}
In Figure \ref{fig:errs}, the evolution of the errors $\mathcal{E}^\alpha_h(t)$, $\alpha\in\{d,w,E\}$ , where $\mathcal{E}^d_h(t):=\| d(t,\cdot)-d_h(t,\cdot)\|_{L^2(\Omega)}$, and the other two defined in a similar way, for $h=2^{-10}$ versus time is shown. It appears that after an initial exponential increase, the error increases linearly with time.
\subsection{Initial data developing singularities}
In our second experiment, we compare the approximations computed by \eqref{num:1}--\eqref{num:2} to those obtained with the algorithms from \cite{BFP}.
\begin{figure}[ht] %
  \centering
  \begin{tabular}{lr}
    \includegraphics[width=0.9\textwidth]{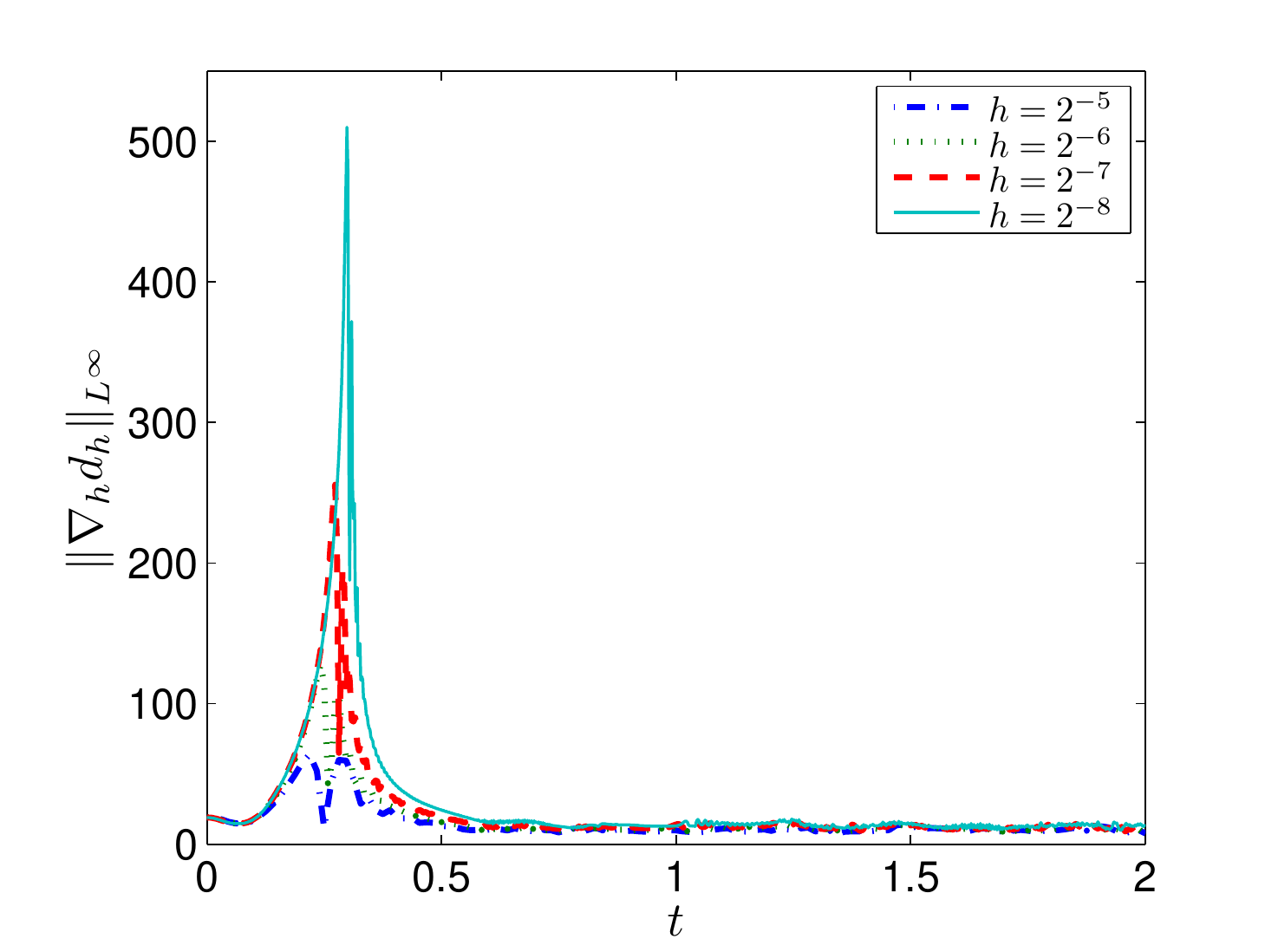}
  \end{tabular}
  \caption{The evolution of $\|\nabla_h d_h\|_{L^{\infty}(\Omega)}$ versus time for initial data \eqref{eq:sb1}.}
  \label{fig:maxgrad}
\end{figure}
Specifically, we compute approximations for the initial data, 
\begin{equation}\label{eq:sb1}
d^0(x,y)=\begin{cases} (0,0,-1)^T, & r\geq 1/2,\\
	(2 x a, 2 y a,a^2-r^2)^T/(a^2+r^2), & r<1/2,
\end{cases}, \quad w^0(x,y)\equiv 0,
\end{equation}
on $D=[-0.5,0.5]^2$, where $r:=\sqrt{x^2+y^2}$ and $a(r)=(1-2 r)^4$ up to time $T=2$, with CFL-condition $\Dt=0.5 h$, $h=2^{-j}$ for $j=5,6,7,8$ and tolerance $\epsilon_0=h^2$. As in \cite{BFP}, we observe a blow-up of the gradient $\nabla d$ in the $L^\infty$-norm around time $T=0.3$, cf. Figure \ref{fig:maxgrad}.
\begin{figure}[ht] %
  \centering
  \begin{tabular}{lr}
    \includegraphics[width=0.5\textwidth]{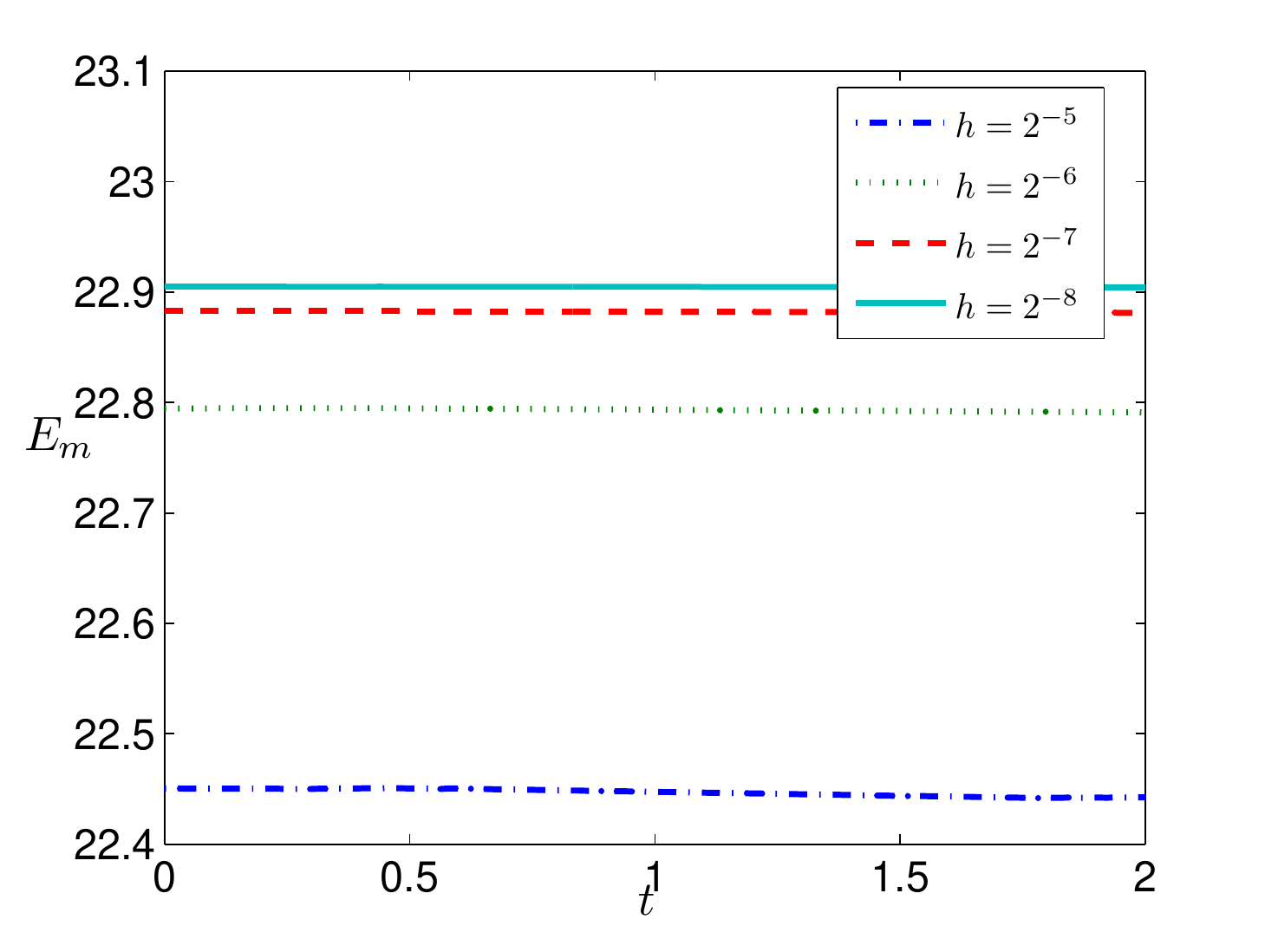}
    \includegraphics[width=0.5\textwidth]{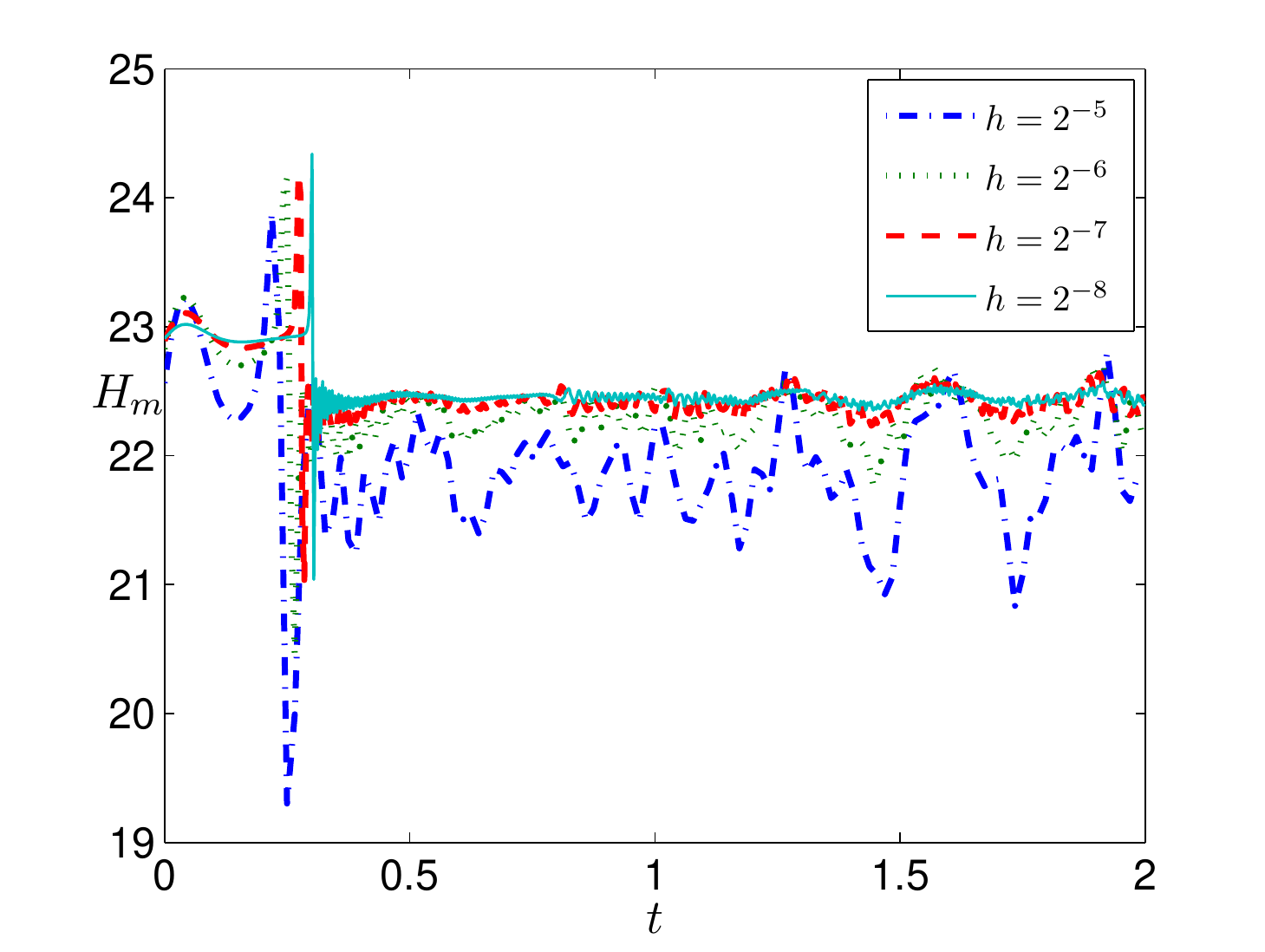}
  \end{tabular}
  \caption{The evolution of $E_m$ and $H_m$ versus time for initial data \eqref{eq:sb1}.}
  \label{fig:EHm}
\end{figure}
 The approximation ${E}_m^{\overline{s}}$ of the discrete energy \eqref{eq:defnrgy} is close to being preserved, as we see in Figure \ref{fig:EHm}, left hand side. In the same figure, on the right hand side, we have plotted the quantity
\begin{equation*}
H_m:=\frac{1}{2}\int_\Omega |D_t d^m_h|^2+|\Grad_h d^m_h|^2\, dx,
\end{equation*}
which is not conserved by our scheme, but upper bounded. We observe some larger oscillations around the time of blow-up of $\Grad_h d_h$ in the $L^\infty$-norm.
\begin{figure}[ht] %
  \centering
  \begin{tabular}{lr}
    \includegraphics[width=0.5\textwidth]{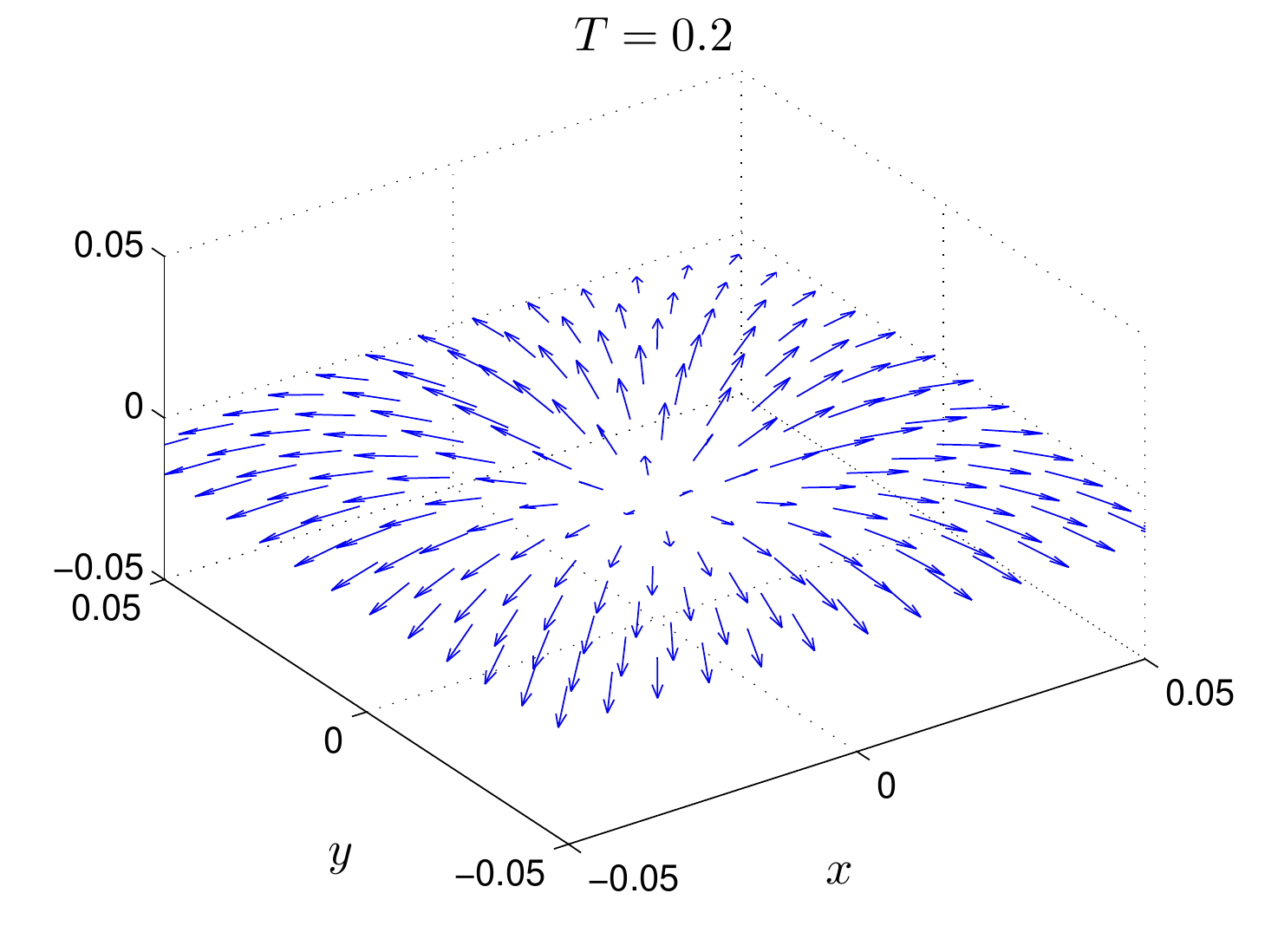}
    \includegraphics[width=0.5\textwidth]{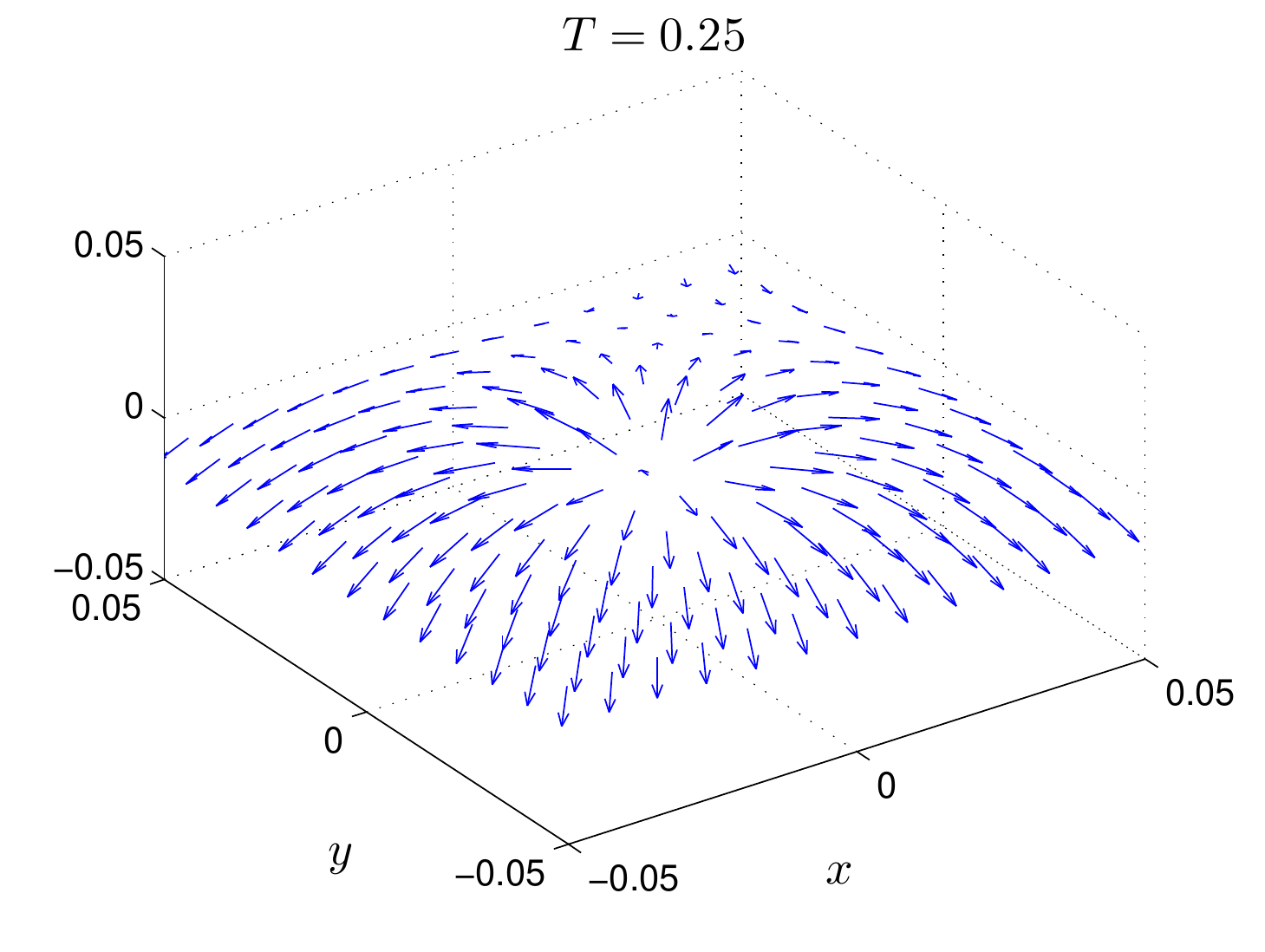}\\
\includegraphics[width=0.5\textwidth]{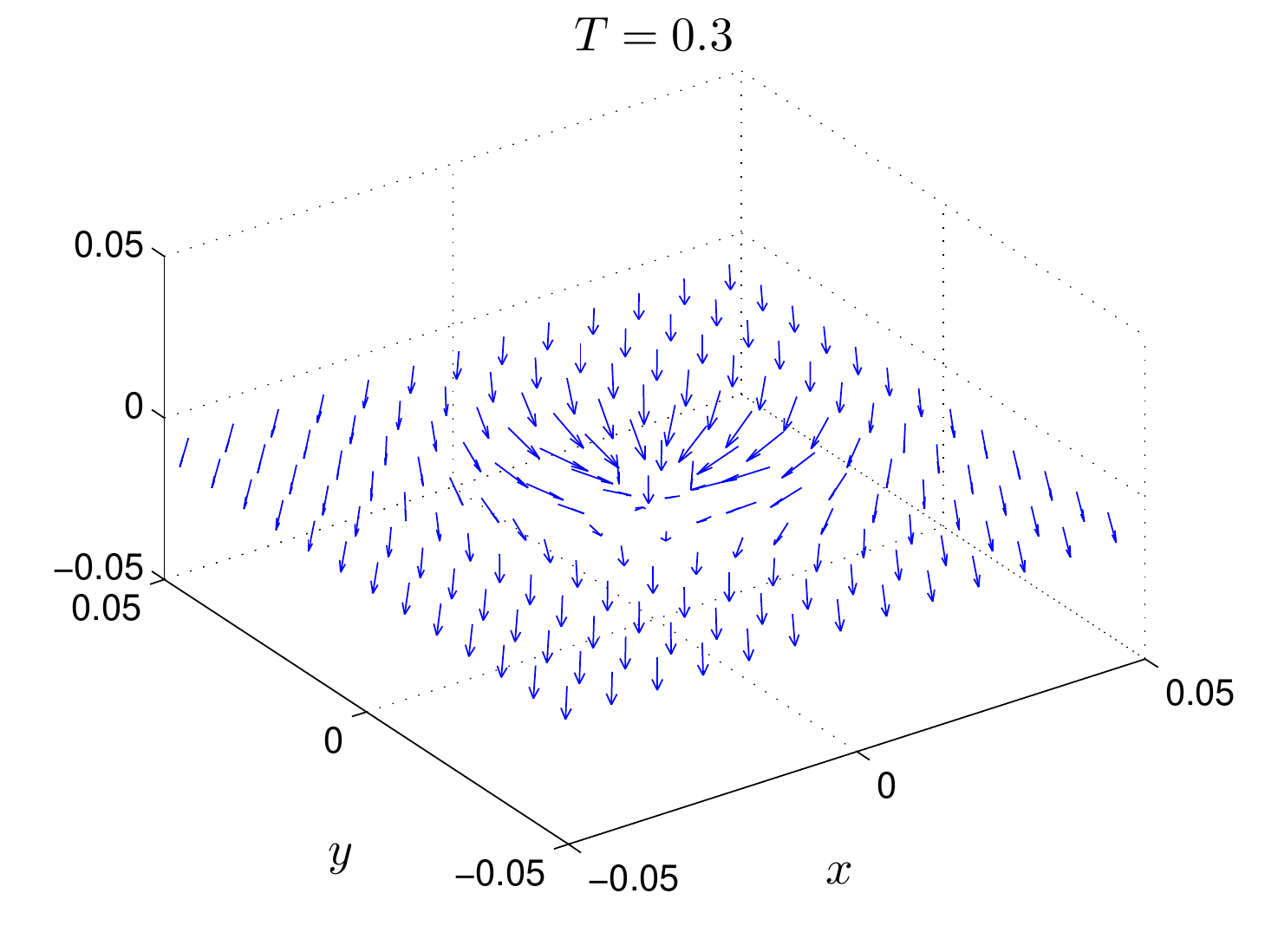}
    \includegraphics[width=0.5\textwidth]{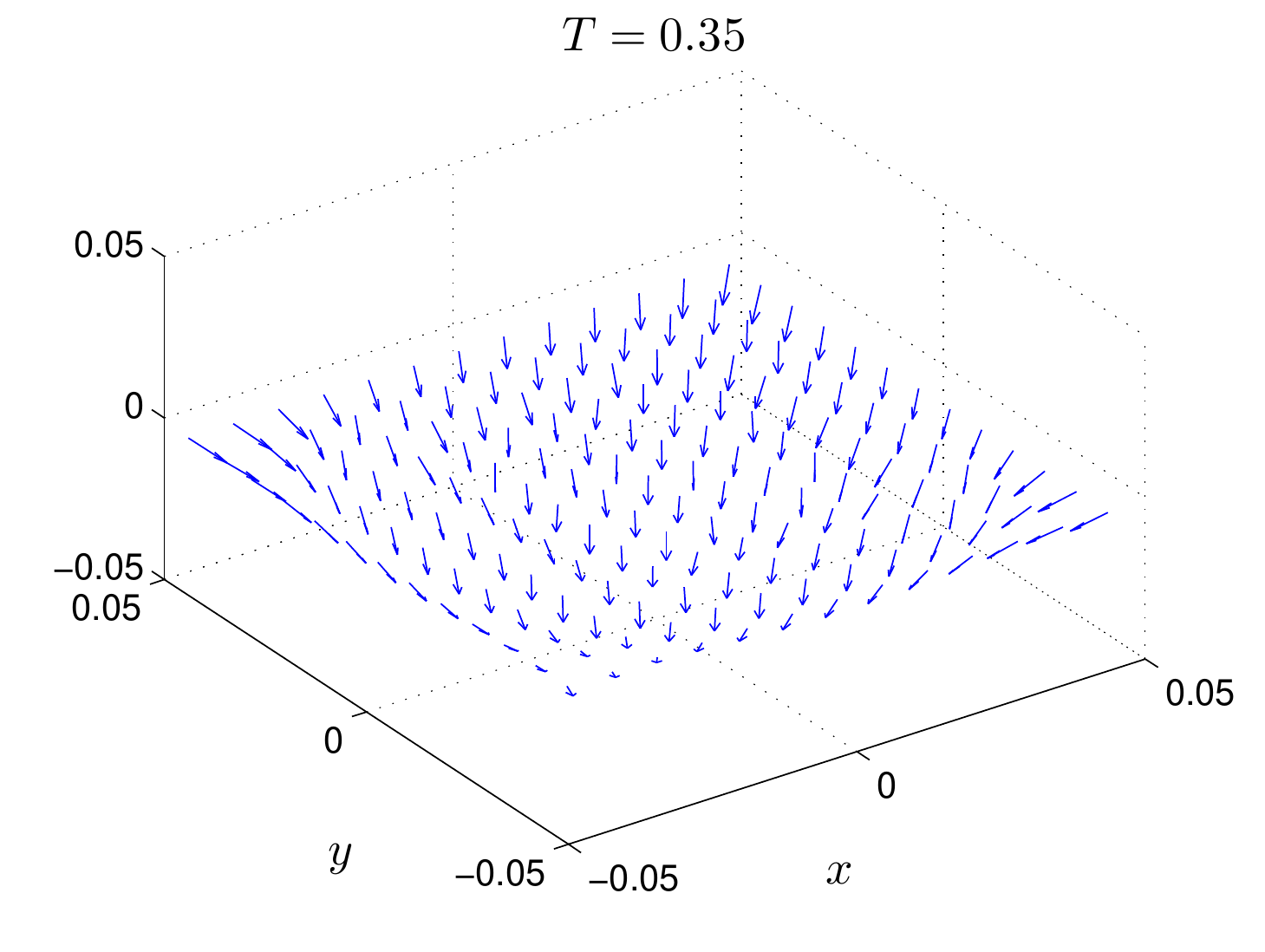}
  \end{tabular}
  \caption{The approximation by \eqref{num:1}--\eqref{num:2} for initial data \eqref{eq:sb1} in a neighborhood of the origin before and after blow-up time on a grid with $h=2^{-7}$.}
  \label{fig:details}
\end{figure}
In Figure \ref{fig:details} the approximation of \eqref{eq:3Dwave}, \eqref{eq:sb1} in a neighborhood of the origin near blow-up time is shown. We observe that the third component of $d$ first switches sign away from the origin and then closer to it, which seems to cause the blow-up in the gradient of $d$.


\begin{thebibliography}{10}
    \bibitem{BPS} {\sc L. Banas, A. Prohl, and R. Schatzle}, 
    {\em Finite element approximations of harmonic map heat flows and wave maps into spheres of nonconstant radii}, 
    Numer. Math., 115, (2010), 395--432.
    
    \bibitem{BLP} {\sc S. Bartels, C. Lubich, and A. Prohl}, 
    {\em Convergent discretization of heat and wave map flows to spheres using approximate discrete Lagrange multipliers}, 
    Math. Comp, 78, (2009), 1269--1292.

    \bibitem{BFP} {\sc S. Bartels, X. Feng, and A. Prohl}, 
    {\em Finite element approximations of wave maps into spheres}, 
    SIAM J. Numer. Anal., 46(1), 2008, 61--87.

    
    \bibitem{B} {\sc S. Bartels}, 
    {\em Semi-Implicit Approximation of Wave Maps into Smooth or Convex Surfaces}, 
    SIAM J. Numer. Anal., 47(5), 2009, 3486--3506.
    
    \bibitem{SS} {\sc J. Shatah and M. Struwe}, 
    {\em Geometric wave equations}, 
    Courant Lecture Notes in Mathematics, 2, 1998.

    
\end{thebibliography}
\end{document}